\documentclass[a4paper,11pt]{amsart}

\usepackage[utf8]{inputenc}
\usepackage{amsfonts,amssymb}
\usepackage{amsmath}
\usepackage{graphicx}
\usepackage{array}
\usepackage{enumitem}
\usepackage{url}
\usepackage{bbm}
\usepackage{comment}
\usepackage{hyperref}
\usepackage{ifthen}
\usepackage{calc}
\usepackage{xargs}

\usepackage{comment}
\usepackage[usenames,dvipsnames,table]{xcolor}

\usepackage{empheq}
\numberwithin{equation}{section}
\usepackage{tikz}
\usetikzlibrary{cd,decorations.pathreplacing,calc,intersections,arrows,arrows.meta}
\tikzstyle{dot}=[shape=circle,draw,color=black,fill=black,inner sep=1pt]

\usepackage{amsthm}
\theoremstyle{plain}

\newtheorem{theorem}{Theorem}[section]
\newtheorem{lemma}[theorem]{Lemma}
\newtheorem{corollary}[theorem]{Corollary}
\newtheorem*{corollary*}{Corollary}

\newtheorem{proposition}[theorem]{Proposition}

\theoremstyle{definition}

\newtheorem*{definition*}{Definition}

\theoremstyle{remark}
\newtheorem{remark}[theorem]{Remark}

\newcommand{\C}{\mathbf{C}}

\newcommand{\Z}{\mathbf{Z}}
\newcommand{\N}{\mathbf{N}}
\newcommand{\cM}{\mathcal{M}}
\newcommand{\cN}{\mathcal{N}}

\newcommand{\cH}{\mathcal{H}}

\DeclareMathOperator{\tr}{tr}

\newcommand{\IIun}{$\mathrm{II}_1$}

\DeclareMathOperator{\Tr}{Tr}

\begin{document}

\title[Orthogonalization of POVMs]{Orthogonalization of Positive Operator Valued Measures}
\date{\today}


\author[M.~De~La~Salle]{Mikael De La Salle}
\address{Université de Lyon, Université Claude Bernard Lyon 1, CNRS, France}
\thanks{MdlS was funded by the ANR grants AGIRA ANR-16-CE40-0022 and Noncommutative analysis on groups and quantum groups ANR-19-CE40-0002-01}

\email{delasalle@math.univ-lyon1.fr}
\begin{abstract}
We show that a partition of the unity (or POVM) on a Hilbert space that is almost orthogonal is close to an orthogonal POVM in the same von Neumann algebra. This generalizes to infinite dimension previous results in matrix algebras by Kempe-Vidick and Ji-Natarajan-Vidick-Wright-Yuen. Quantitatively, our result are also finer, as we obtain a linear dependance, which is optimal.

We also generalize to infinite dimension a duality result between POVMs and minimal majorants of finite subsets in the predual of a von Neumann algebra. 
\end{abstract}
\maketitle
\section{Orthonormalization of partitions of unity in infinite dimensional Hilbert space}

Stability is a term coined in \cite{Ulam} to describe a situation when mathematical objects that almost satisfy certain properties are close to objects exactly satisfying these properties. It has been recently much studied for groups, mainly motivated by the question of whether every group is hyperlinear/sofic. The same phenomena but for quantum strategies for two-player non-local games are also central in the recent work by Ji-Natarajan-Vidick-Wright-Yuen \cite{quantumsoundness,MIPRE}. The goal of this note is to explore a small portion of \cite{quantumsoundness} and discuss its possible generalizations to infinite dimension and some consequence in terms of stability.

The objects that we study in this note are finite families $(t_1,\dots,t_n)$ of positive operators on a complex Hilbert space $\cH$ which sum to the identity: $t_1+\dots+t_n = 1_{\cH}$. They are called \emph{partitions of unity} by operator algebraists, and \emph{Positive Operator Valued Measures} (POVM) by quantum information theorists. We will use POVM here as it is shorter, and the integer $n$ is called the number of outputs. And we talk about Projection Valued Measures (PVM) if in addition all the $t_i$'s are projections ($t_i^2=t_i$). The main result of this note is the following. The question whether a form of this result holds in infinite dimension answers a question asked by Henry Yuen (private communication). This result is used in the subsequent work \cite{quantumsoundness2}.

\begin{theorem}\label{thm:orthonormalization_Hilbert} Let $(a_1,\dots,a_n)$ be a POVM on a complex Hilbert space $\cH$, let $\xi \in \cH$ be a unit vector and $\varepsilon \in [0,1]$ satisfying $ \sum_i \|a_i \xi\|^2 > 1-\varepsilon$.

There exists an orthogonal decomposition $\cH = \cH_1\oplus \cH_2 \oplus \dots\oplus \cH_n$ such that
\begin{enumerate}
\item\label{item:Ai_close_to_Pi} if $\xi_i$ denotes the orthogonal projection on $\cH_i$, then $\sum_i \|a_i\xi - \xi_i\|^2 < 9 \varepsilon$,
\item\label{item:commutant} every operator $b \in B(\cH)$ which commutes with each $a_i$ preserves each $\cH_i$.
\end{enumerate}
\end{theorem}
In other words, the condition $\sum_i \|a_i \xi\|^2 > 1-\varepsilon$ implies that: \eqref{item:Ai_close_to_Pi}  $(a_1,\dots,a_n)$ is $9\varepsilon$-close on $\xi$ to a PVM $(p_1,\dots,p_n)$ \eqref{item:commutant} which preserves the symmetries of the original POVM.

This theorem generalizes to infinite dimension and slightly strengthens a result from \cite{quantumsoundness} (see also \cite[Lemma 19]{zbMATH06301159}), which proves a form of this theorem with finite dimensional Hilbert spaces. More precisely, we can rephrase \cite[Theorem 5.2]{quantumsoundness} as follows: if $\cH = \cH_A \otimes \cH_B$ is a tensor product of \emph{finite dimensional} Hilbert spaces and the $a_i$ are of the form $A_i \otimes \mathrm{id}_{B}$, then the theorem holds, but with $9\varepsilon$ in \eqref{item:Ai_close_to_Pi} (which is essentially optimal, see Remark~\ref{rem:9epsilon_optimal}) replaced by $100 \varepsilon^{\frac 1 4}$. Formally, the conclusion \eqref{item:commutant} is replaced by a slightly weaker conclusion, as it only requires that the orthogonal projections on each factor in the orthogonal decomposition are of the same tensor product form $P_i \otimes \mathrm{id}_{B}$, or equivalently that the decomposition is of the form 
\[ \cH = (\cH_1\otimes \cH_B) \oplus \dots \oplus (\cH_n \otimes \cH_B).\]
But the main contribution here is to deal with infinite dimensional Hilbert spaces.

If one only requires the conclusion \eqref{item:Ai_close_to_Pi}, Theorem~\ref{thm:orthonormalization_Hilbert} becomes a small exercise in euclidean geometry, and the dimension of $\cH$ is irrelevant as everything happens in the space spanned by the $n$ vectors $a_i \xi$. It is the conclusion \eqref{item:commutant} that makes the statement dependant on the dimension of $\cH$ and of the structure of the algebra of operators commuting with $a_i$. So, although it is stated in Hilbert-space vocabulary,  Theorem~\ref{thm:orthonormalization_Hilbert} is a result about von Neumann algebras and states. This paper might be read by non-experts in von Neumann algebras, so we will try to recall basic definitions in section~\ref{sec:reminders_on_vnA}, and to give complete proofs or precise references for the statements we need. We refer to standard textbooks such as \cite{Takesaki1} for more background.


The following is an equivalent reformulation of Theorem~\ref{thm:orthonormalization_Hilbert} in von Neumann algebraic language.
\begin{theorem}\label{thm:orthonormalization} Let $\cM$ be a von Neumann algebra with a normal state $\varphi$, and let $(a_i)$ be a POVM in $\cM$ such that $\varphi(\sum_i a_i^2) > 1-\varepsilon$. 

Then there is a PVM $(p_i) \subset \cM$ (made of projections) such that $\varphi(\sum_i |a_i - p_i|^2) < 9 \varepsilon$.
\end{theorem}
\begin{remark}\label{rem:9epsilon_optimal}  Conversely, if there is a PVM $(p_i)$ such that $\sum_i \varphi(|a_i - p_i|^2) \leq \delta$, then by the triangle inequality
 \[ \varphi(\sum\nolimits_i a_i^2) \geq \bigg(1- \sqrt{\varphi(\sum\nolimits_i |a_i - p_i|^2)}\bigg)^{2} \geq 1-2\sqrt{\delta}.\]
This could have suggested that the upper bound $9\varepsilon$ in Theorem~\ref{thm:orthonormalization} is not optimal and can be replaced by $O(\varepsilon^2)$. This is not the case, and the $9\varepsilon$ cannot be replaced by anything smaller than $\varepsilon$, as the following simple example illustrates.

Consider $\cM=\ell_\infty^2$ ($\C^2$ with the $\ell_\infty$ norm), $a_1=(1,\frac 1 2)$ and $a_2 = (0,\frac 1 2)$, and $\varphi$ is the state $\varphi(x,y) = (1-c)x+c y$, then we have 
\[ \varphi(a_1^2+a_2^2) = 1-\frac 1 2 c,\]
and for every PVM $(p_1,p_2)$ we have 
\[ \varphi(|p_1 - a_1|^2) + \varphi(|p_2-a_2|^2) \geq \frac 1 2 c.\]
\end{remark}
Before we prove the main Theorem~\ref{thm:orthonormalization}, let us state one consequence, which says that almost commuting PVMs are close to commuting PVMs.

Let us denote, for an element $a$ in a von Neumann algebra and a normal state $\varphi$,  $\|a\|_\varphi = \sqrt{\varphi(a^*a)}$.
\begin{theorem}\label{thm:pvm_almost_commute} Let $(p_i)_i$ and $(q_j)_j$ be two PVMs in a von Neumann algebra $\cM$ and $\varphi$ be a normal state on $\cM$. If $\sum_{i,j} \|p_i q_j - q_j p_i\|_\varphi^2 < \varepsilon$, then there is another PVM $(p_i')_i$ in $\cM$ such that $[p'_i,q_j]=0$ for every $i,j$ and \[\sum_i \|p_i - p'_i\|_\varphi^2 < 10 \varepsilon.\]
\end{theorem}
By Fourier transform (Pontryagin duality), PVM's with $n$ outputs are in one-to-one correspondence with unitaries $u$ of order $n$ ($u^n=1$): to $(p_1,\dots,p_n)$ corresponds $u=\sum_{k=1}^n e^{\frac{2ik\pi}{n}} p_k$. The inverse maps $u$ to $(p_1,\dots,p_n)$ where $p_j = \frac{1}{n} \sum_k e^{-\frac{2ijk\pi}{n}} u^k$. Therefore, the previous theorem is formally equivalent to the following. It is a new form of a statement asserting that almost commuting unitaries are close to unitaries, that does not seem to be comparable with existing results, even when $\varphi$ is a trace. Comparing with the probabilistic results in \cite{BeckerChapman} for permutation actions (that Michael Chapman kindly pointed out to me) raises the question whether there is a form of Corollary~\ref{cor:almost_commuting} that is valid for arbitrary amenable groups and not just $\Z/n\Z\times \Z/m\Z$.
\begin{corollary}\label{cor:almost_commuting} Let $(\cM,\varphi)$ be a von Neumann algebra with a normal state, and $u,v \in \cM$ be unitaries of finite order $n,m$.

If $\frac{1}{nm} \sum_{i=1}^n \sum_{j=1}^m \| u^i v^j - v^j u^i\|_\varphi^2 < \varepsilon$, then there is a unitary $v' \in \cM$ that commutes with $u$ and satisfies
\[ \frac{1}{n}\sum_i \| v^i - {v'}^i\|_\varphi^2 < 10 \varepsilon.\]
\end{corollary}

The proof of the main result is not very involved and follows the same general strategy as in \cite{quantumsoundness}, but it requires a bit of familiarity with von Neumann algebras and some adaptations to obtain the optimal order in the constants. We present the necessary background in Section~\ref{sec:reminders_on_vnA}, and then prove the Theorem by decomposing it into three different cases. We deduce Theorem~\ref{thm:pvm_almost_commute} in Section~\ref{sec:AC}. Finally, Section~\ref{sec:HB} generalizes to infinite dimensional von Neumann algebras a semidefinite program considered in \cite{quantumsoundness}.

\subsection*{Acknowledgements} I thank Thomas Vidick and Henry Yuen their patience answering my questions, and for encouraging me to write down this note. Thanks also to Michael Chapman, Gilles Pisier and Thomas Vidick for comments and corrections on preliminary versions of this note. Finally, thanks are due to Gilles Pisier and the referee who pointed out mistakes in the initial proof of Proposition~\ref{prop:minmax} and in my dealing of separability issues respectively. I also thank Amine Marrakchi for many interesting discussions.
\section{Facts on von Neumann algebras}\label{sec:reminders_on_vnA}

A von Neumann algebra is a self-adjoint subalgebra of the algebra $B(\cH)$ of bounded operators on a complex Hilbert space $\cH$ that is equal to its bicommutant. Here the \emph{commutant} of a subset $F \subset B(\cH)$ is the  algebra $F'$ of operators that commute with all elements of $F$, and its \emph{bicommutant} is the commutant of its commutant. The von Neumann bicommutant theorem \cite[Theorem II.3.9]{Takesaki1} is a fundamental result of the theory, which asserts that the bicommutant of a self-adjoint subset $F \subset B(\cH)$ coincides with the weak-* closure of the self-adjoint unital algebra generated by $F$, where we see $B(\cH)$ as the dual of the trace-class operators on $\cH$. In particular, a von Neumann algebra is a dual space. Another fundamental theorem \cite[Corollary III.3.9]{Takesaki1} asserts that a von Neumann algebra admits a unique predual. This allows to talk about \emph{the} weak-* topology on $\cM$; it coincides with the ultraweak operator topology, the smallest topology making continuous all linear maps of the form $x \mapsto \sum_k \langle x \xi_k,\eta_k\rangle$ for sequences $\xi_k,\eta_k \in \cH$ satisfying $\sum_k \|\xi_k\| \|\eta_k\|<\infty$. 

A \emph{normal state} on a von Neumann algebra $\cM \subset B(\cH)$ is a linear map $\varphi \colon \cM \to \C$ that is positive ($\varphi(x^*x) \geq 0$ for every $x \in \cM$), normalized by $\varphi(1)=1$ and weak-* continuous. The typical example of a state is a vector state $x \mapsto \langle a\xi,\xi\rangle$ for a unit vector $\xi \in \cH$. The Gelfand-Naimark-Segal (GNS) construction asserts that (up to changing the Hilbert space), every normal state can be realized as a vector state.

We will prove the theorem by a reduction to two main cases (finite and type III). To state the reduction we need to recall some basic definitions on the type of a von Neumann algebra (see \cite[Chapter V]{Takesaki1}). Throughout this note, by projection we always mean self-adjoint projection: $p=p^*=p^2$. Given a von Neumann algebra $\cM$, we say that

\begin{itemize}
\item $\cM$ is finite if for every $u \in \cM$, $u^*u=1$ implies $uu^*=1$. A projection $p \in \cM$ is finite if the von Neumann algebra $p\cM p$ is finite.
\item $\cM$ is of type \IIun{} if it is finite and if $0$ is the only projection $p \in \cM$ such that $p\cM p$ is commutative.
\item $\cM$ is semi-finite if every nonzero projection $p \in \cM$ majorizes a non-zero finite projection.
\item $\cM$ has type III if it does not contain any nonzero finite projection.
\end{itemize}

We know from general theory \cite[Theorem V.1.19]{Takesaki1} that every von Neumann algebra $\cM$ can be written as a direct sum of a semifinite and a type III von Neumann algebra. So it is enough to separately prove Theorem~\ref{thm:orthonormalization} when $\cM$ is semi-finite and when $\cM$ has type III. The semi-finite case can easily be reduced to the finite case, which is the most interesting one.

\section{Proof of Theorem~\ref{thm:orthonormalization} when $\cM$ is finite}
Let $\cM,\varphi,(a_i)$ be as in Theorem~\ref{thm:orthonormalization}, with $\cM$ finite. By \cite[Theorem V.2.6]{Takesaki1}, the finiteness assumption is equivalent to the existence of a \emph{normal center-valued trace}, that is a normal conditional expectation $E\colon \cM \to Z(\cM)$ onto the center of $\cM$ such that $E(ab) = E(ba)$ for every $a,b \in \cM$. 

The first lemma contains all the difficulty in the proof of Theorem~\ref{thm:orthonormalization}.
\begin{lemma}\label{lem:selection_of_projections} There are projections $q_i$ commuting with $a_i$ such that 
\begin{equation}\label{eq:sum_rank} \sum_{i=1}^n E(q_i) = 1,
\end{equation}
\begin{equation}\label{eq:large_trace}
\varphi\left(\sum_{i=1}^n q_i a_i\right)\geq 1-\varepsilon. 
\end{equation}
\end{lemma}
Before we prove the Lemma, let us observe that it is not possible to replace \eqref{eq:sum_rank} by the stronger condition $\sum_i q_i=1$. Indeed, if $\cM = M_2(\C)$ and if the $a_i's$ do not have any common eigenvector, then the only families $(q_i)$ of projections commuting with $a_i$ satisfying $\sum_i q_i=1$ are when one of the $q_i$'s is the identity and the other are $0$, and so \eqref{eq:large_trace} would become $\max_i \varphi(a_i) \geq 1 - \varepsilon$. A concrete example is given by $n=3$, $\varphi$ the normalized trace on $M_2(\C)$ and 
\begin{align*} a_1 &= \frac{1}{1+6\delta}\begin{pmatrix} 1+4\delta & 0 \\0&0\end{pmatrix},\\ a_2 &= \frac{1}{1+6\delta}\begin{pmatrix} \delta & \sqrt{3}\delta \\ \sqrt{3}\delta & 1+3 \delta\end{pmatrix},\\ a_3 &= \frac{1}{1+6\delta}\begin{pmatrix} \delta & -\sqrt{3}\delta \\ -\sqrt{3}\delta & 3 \delta\end{pmatrix}. \end{align*}
This example satisfies $\varphi(\sum_i a_i^2) \geq 1 - 5 \delta + o(\delta)$, but $\max_i \varphi(a_i) \leq \frac 1 2$.

\begin{proof} Consider the subset $C \subset\cM^n$
\[ C = \{(x_1,\dots,x_n) \in \cM^n \mid \forall i, 0 \leq x_i \leq 1, x_ia_i = a_i x_i, \sum\nolimits_i E(x_i)=1\}.\]
$C$ is a convex subset in the unit ball of $\cM^n$, and contains $(a_1,\dots,a_n)$. It is clearly weak-* closed, and therefore compact as the unit ball of $\cM^n$ is weak-* compact. By the Krein-Milman theorem, $C$ is the closure of the convex hull of its extreme points, and in particular the continuous affine map
\[ f \colon (x_1,\dots,x_n) \in C \mapsto \varphi(\sum\nolimits_i x_i a_i)\]
attains its maximum (which is $\geq f(a_1,\dots,a_n) \geq 1-\varepsilon$) at an extreme point. 

So all we have to do is to show that
\begin{equation}\label{eq:claim} \textrm{if $(x_1,\dots,x_n)$ is an extreme point of $C$, then each $x_i$ is a projection.}
\end{equation}

We know from general theory \cite[Theorem V.1.19 and V.1.27]{Takesaki1} that there is a sequence $(z_d)_{d \in \N}$ of orthogonal projections in $Z(\cM)$ such that $z_d \cM$ is isomorphic to $M_d(\C) \otimes z_d Z(\cM)$ and $(1-\sum_d z_d) \cM$ is of type \IIun{}. If $x_i$ was not a projection, then either there is $d \in \N$ such that $z_d x_i$ is not a projection, or $(1-\sum_d z_d) x_i$ is not a projection. So we are reduced to showing \eqref{eq:claim} when $\cM = M_d(\C) \otimes Z$ for an abelian von Neumann algebra $Z$, or when $\cM$ is of type \IIun{}. 

The latter case is easier, as the extremality of $(x_1,\dots,x_n)$ in $C$ in particular implies that, for each $i$, $x_i$ is extremal inside 
\[ C_i := \{ y_i \in \cM \mid 0 \leq y_i \leq 1, y_i a_i = a_i y_i, E(y_i)=E(x_i)\},\] and this weaker condition already implies that $x_i$ is a projection. Indeed, if $x_i$ was not a projection, there would exist $\delta>0$ such that the spectral projection $p:=\chi_{[\delta,1-\delta]}(x_i)$ is nonzero and commutes with $a_i$. By the definition of $\cM$ being of type \IIun{}, we know that $p\cM p$ is not abelian. In particular, there is a self-adjoint element $b \in p\cM p$ which commutes with $pa_i$ but does not belong to $p Z(\cM)$ (we can take $b=pa_i$ if $p a_i \notin pZ(\cM)$ and an arbitrary selfadjoint element of $p \cM p \setminus p Z(\cM)$ otherwise). We can moreover assume that $0 \leq b \leq p$. Using that the center valued trace $E: \cM \to Z(\cM)$ is (completely) positive, we obtain that $0 \leq E(b) \leq E(p)$, and in particular there is $z \in Z(\cM)$ such that $E(b) = z E(p)$ and $0 \leq z \leq 1$. Then $b':=b-zp$ is a nonzero selfadjoint element of $(\{p a_i\}' \cap p\cM p)\setminus p Z(\cM)$, which moreover satisfies $E(b')=E(b) - zE(p) = 0$. It has norm $\leq 1$. So we can write $x_i$ as the midpoint between $x_i +\delta b'$ and $x_i - \delta b'$, which both belong to $C_i$. This contradicts the extremality of $x_i$ in $\cM$ and proves \eqref{eq:claim} when $\cM$ is of type \IIun{}.

In the first case, we can write equivalently $\cM = M_d(\C) \otimes
L_\infty(\Omega,\mu)$ for a measure space $(\Omega,\mu)$ \cite[Theorem
  III.1.18]{Takesaki1}. We first consider the simpler situation when $\cM = M_d(\C)$. Then \eqref{eq:claim} becomes that any
extreme point in $\{ (x_1,\dots,x_n) \in M_d(\C) \mid 0 \leq x_i \leq
1, [x_i,a_i]=0, \Tr(\sum_i x_i)=d\}$ is made of projections. Assume
for a contradiction that this is not the case, and that
$(x_1,\dots,x_n)$ is an extreme point not entirely made of
projections. If, for some $i$, $x_i$ has at least two nonzero
eigenvalues different from $0$ and $1$ (counting with multiplicities),
then we can do as in the \IIun{} case, choose orthogonal rank one
projections $p_1,p_2$ corresponding to these eigenvalues of $x_i$ and
commuting with $a_i$, and for $\delta>0$ small enough the
decomposition \[ x_i = \frac{1}{2}\left((x_i+\delta p_1-\delta p_2) +
(x_i-\delta p_1+\delta p_2)\right)\] will contradict the
extremality. Otherwise, using that $\sum_k \Tr(x_k)$ is an integer,
there are at least two indices $i \neq j$ such that $x_i$ and $x_j$
both have exactly one eigenvalue not in $\{0,1\}$, counting
multiplicities. In that case, if $p_i$ and $p_j$ are the corresponding
rank one projections, they necessarily commute with $a_i$ and $a_j$
respectively, and we can define for $\delta \in [-1,1]$
\[ x_k(\delta) = \begin{cases} x_i + \delta p_i& \textrm{if }k=i\\
x_j - \delta p_j& \textrm{if }k=j\\
x_k & \textrm{if }k\notin{i,j}.
\end{cases}.\]
For $|\delta|$ small enough $(x_1(\delta),\dots,x_n(\delta))$ belongs to $C$, and the expression $x_k = \frac{1}{2}(x_k(\delta)+x_k(-\delta) )$ also contradicts the extremality of $(x_1,\dots,x_n)$.

To summarize, when $\cM=M_d(\C)$ we have constructed, for every $x \in C$ that is not made of projections, two distinct points $x',x'' \in C$ such that $x=\frac{1}{2}(x'+x'')$. An inspection of the proof reveals that the map $x \mapsto (x',x'')$ can be made Borel-measurable. As a consequence, the proof applies also to $\cM = M_d(\C) \otimes L_\infty(\Omega,\mu)$. 
\end{proof}
We shall use the following elementary fact about finite von Neumann algebras.
\begin{lemma}\label{lem:polar_decom_finite} Let $\cM$ be a finite von Neumann algebra with center-valued trace $E$, and $x \in \cM$. If $p$ and $q \in \cM$ are projections such that $E(p) = E(q)$ and
  \[ xp = qx=x,\]
  then we can decompose $x=u |x|$ where $u^*u=p$ and $uu^*=q$.
\end{lemma}
\begin{proof} Denote by $p_0 \in \cM$ the left support of $x$ and $q_0\in \cM$ the right support of $x$ (that is, if $\cM \subset B(\cH)$, $p_0$ and $q_0$ are the smallest projection in $B(\cH)$ such that $xp_0=x$ and $q_0 x=x$ respectively). Write $x= u_0 |x|$ be the usual polar decomposition of $x$, where $|x| = (x^* x)^{1/2}$, and $u_0 \in \cM$ is a partial isometry with $u_0^* u_0=p_0$ and $u_0^* u_0 = q_0$ (see \cite[Proposition 2.2.4]{AnantharamanPopa}). By definition, we have $p_0 \leq p$ and $q_0 \leq q$. Moreover,
  \[E(p-p_0) = E(p) - E(u_0^* u_0) = E(q) - E(u_0 u_0^*) = E(q-q_0).\]
  By \cite[Proposition 9.1.8]{AnantharamanPopa}, we have that $p-p_0 \sim q-q_0$. That is, there is a partial isometry $v$ such that $v^*v = p-p_0$ and $vv^* = q-q_0$. The lemma holds with $u=v+u_0$.
  \end{proof}
We can now prove the main result. With Lemma~\ref{lem:selection_of_projections} and Lemma~\ref{lem:polar_decom_finite} in hand, the proof is very close to \cite{quantumsoundness}.
\begin{proof}[Proof of Theorem~\ref{thm:orthonormalization} when $\cM$ is finite]
Let $q_i$ be given by Lemma~\ref{lem:selection_of_projections}. Consider the matrix $x = \sum_i e_{i,1} \otimes q_i a_i^{1/2}$, that we see in $M_n(\cM)$. What is important for us is that $M_n(\cM)$ is a finite von Neumann algebra. Specifically, its center is $1_n\otimes Z(\cM)$ and the corresponding central-valued trace is 
\[E_n:(a_{i,j}) \mapsto 1_n \otimes (\frac{1}{n} \sum\nolimits_i E(a_{i,i})).\]
Let $p = \sum_i e_{i,i} \otimes q_i$ and $q=e_{1,1}\otimes 1$. Then $E_n(p) = 1_n \otimes \frac{1}{n} \sum_i E(q_i) = \frac{1}{n} = E_n(q)$, so by Lemma~\ref{lem:polar_decom_finite} we can write $x=u|x|$ with $uu^* = \sum_i e_{i,i} \otimes q_i$ and $u^* u = e_{1,1}\otimes 1$. In the following, we identify $\cM$ with $\{e_{1,1}\otimes a\mid a \in \cM\}$. If $t_i = e_{i,i}\otimes q_i$, then $u^* t_i u$ is a projection $p_i$ in $\cM$ (formally it is of the form $e_{1,1} \otimes p_i$, but we decided to identify this with $p_i$) for projections $p_i \in \cM$ which sum to $1$. Moreover, we have
\[ |x| p_i |x| = |x| u^* t_i u |x| = x^* t_i x = q_i a_i.\]
Let us denote by $\|\cdot\|_\varphi$ the norm on $\cM^n$ given by \[ \| (B_i)_i \|_\varphi^2 = \sum\nolimits_{i=1}^n \varphi(B_i^* B_i).\]
By the triangle inequality, decomposing $a_i - p_i = a_i - q_i a_i + (|x|-1) p_i |x| + p_i(|x|-1)$, we obtain
\[ \| (a_i - p_i)_i \|_\varphi \leq \|(a_i - q_i a_i)_i\|_\varphi+\|((1-|x|) p_i |x|)_i\|_\varphi + \| (p_i (1-|x|))_i\|_\varphi.\]
We shall bound each term. The first term is easy:
\[\|(a_i - q_i a_i)_i\|_\varphi^2 = \sum_i \varphi((1-q_i) a_i^2) \leq \sum_i \varphi((1-q_i) a_i) \leq \varepsilon.\]
The third term is also easy:
\[ \| (p_i (1-|x|))_i\|_\varphi^2 = \sum_i \varphi( (1-|x|)p_i(1-|x|)) = \varphi( (1-|x|)^2) \leq \varphi(1-|x|^2) \leq \varepsilon.\]
We used that $(1-|x|)^2 \leq 1-|x|^2$, which is true because $0 \leq |x| \leq 1$. For the second term we proceed similarly but with more care:
\begin{align*}
\|((1-|x|) p_i |x|)_i\|_\varphi^2 &= \sum_i \varphi(|x| p_i(1-|x|)^2 p_i |x|)\\
& \leq \sum_i \varphi(|x| p_i(1-|x|^2) p_i |x|)\\
& =\sum_i \varphi(|x| p_i |x| - (|x| p_i |x|)^2).\\
\end{align*}
Using that $|x| p_i |x| = q_i a_i$, we obtain
\[\|((1-|x|) p_i |x|)_i\|_\varphi^2 \leq\sum_i  \varphi(q_i(a_i-a_i^2)) \leq \sum_i (a_i - a_i^2) \leq\varepsilon.\]
To conclude, we obtain
\[ \| (a_i - p_i)_i \|_\varphi \leq 3 \sqrt{\varepsilon},\]
which is the desired conclusion.
\end{proof}

\section{Proof of Theorem~\ref{thm:orthonormalization} when $\cM$ is semi-finite}
We  deduce easily the case when $\cM$ is semi-finite from the finite case, thanks to the following basic fact.
\begin{lemma}\label{lem:semifinite-to-finite} If $(\cM,\varphi)$ is a von Neumann algebra with a normal state and $(p_\alpha)$ is a net of projections tending to $1$, then for every finite family $C_1,\dots,C_k \in \cM$, we have
  \[\lim_\alpha \varphi(p_\alpha  C_1 p_\alpha C_2 \dots p_\alpha C_k p_\alpha) = \varphi(C_1 C_2\dots C_k).\]
\end{lemma}
\begin{proof} This follows simply from the fact that, on the unit ball $B$ of $\cM$ equipped with the strong operator topology, multiplication $B \times B \to B$ is continuous.


\end{proof}

So let $\cM,\varphi,(a_i)$ be as in Theorem~\ref{thm:orthonormalization}, with $\cM$ semi-finite. By \cite[Theorem V.1.37]{Takesaki1}), there is an increasing net $p_\alpha$ of finite projections in $\cM$ such that $\lim_\alpha p_\alpha = 1$. Let $\varphi_\alpha$ be the state $\frac{1}{\varphi(p_\alpha)}\varphi$ on $\cM_\alpha:=p_\alpha \cM p_\alpha$, and define a POVM $a_{i,\alpha} = p_\alpha a_i p_\alpha$ in $\cM_\alpha$. It follows from Lemma~\ref{lem:semifinite-to-finite} that for every $\alpha$ large enough, $\varphi_\alpha( \sum_i a_{i,\alpha}^2) >\varepsilon$, and the (already proven) finite case of Theorem~\ref{thm:orthonormalization} provides us with a PVM $p_{i,\alpha}$ in $p_\alpha \cM p_\alpha$ satisfying the conclusion of the theorem. Then for $\alpha$ large enough, the PVM $(1-p_\alpha + p_{1,\alpha},p_{2,\alpha},\dots,p_{n,\alpha})$ satisfies the conclusion of Theorem~\ref{thm:orthonormalization}.

\section{Proof of Theorem~\ref{thm:orthonormalization} when $\cM$ is type III}
The type III case will be proven with the same strategy as the finite case, but the details are simpler, and the constants are a bit better (the $9$ can be replaced by $1$ in that case). We shall need the following. Recall that the central support of an element $x \in \cM$ is the smallest projection $z(x)$ in the center of $\cM$ such that $z(x)x=x$. We shall use the standard terminology on comparision of projections \cite[Chapter V]{Takesaki1}: we say that two projections $p,q \in \cM$ are equivalent and write $p \sim q$ if there is $u \in \cM$ such that $u^* u =p$, $uu^*=q$. We write $p \prec q$ if there is a projection $p'$ equivalent to $p$ such that $q-p'$ is positive.
\begin{lemma}\label{lem:typeIII} If $\cM$ is a type III von Neumann algebra, then there is a net $(q_\alpha)$ of projections in $\cM$ converging weak-* to $1$ and such that $1-q_\alpha \sim 1$ for every $\alpha$.
\end{lemma}
\begin{proof} Let $\psi \in \cM_*$ be a state. In the first step of the proof, we shall construct an increasing sequence $(q_k)_{k \in \N}$ of projections such that $\psi(q_k) \geq 1-2^{-k}$ and $(1-q_k) \sim 1$. The construction is by induction. Define $q_0=0$. If $q_k$ is defined, then $(1-q_k) \cM (1-q_k)$ is of type III, so by \cite[Proposition V.1.36]{Takesaki1}, there is projection $e_k \in (1-q_k) \cM (1-q_k)$ such that, if $f_k =1-q_k-e_k$, then $e_k \sim f_k \sim 1-q_k$. In particular, both $f_k$ and $e_k$ are equivalent to $1$ in $\cM$. Moreover, we have $\psi(e_k) + \psi(f_k) = \psi(1-q_k) \leq 2^{-k}$, so $\min(\psi(e_k), \psi(f_k)) \leq 2^{-k-1}$. It remains to define $q_{k+1} = 1-e_k$ if $\psi(e_k) \leq \psi(f_k)$ and $q_{k+1} = 1-f_k$ otherwise.

This sequence $q_k$ depends on $\psi$, so let us denote it $q_{k,\psi}$.

Consider the set $A$ of all finite sets of normal states on $\cM$, ordered by inclusion. For every $\alpha = \{\psi_1,\dots,\psi_d\} \in A$, we define $q_\alpha$ to be $q_{d,\psi}$ where  $\psi = \frac{1}{d}\sum_{i=1}^d\psi_i$. It satisfies $\psi_i(1-q_\alpha) \geq 1-d 2^{-d}$ for every $1 \leq i\leq d$. In other words, the net $(q_\alpha)$ satisfies $\lim_\alpha \psi(q_\alpha)=1$ for every normal state $\psi$. That is, it converges to $1$ weak-*.
\end{proof}

So let $\cM,\varphi,(a_i)$ be as in Theorem~\ref{thm:orthonormalization}, with $\cM$ type III, and let $q_\alpha$ be as in the previous lemma. Consider $v_\alpha = \sum_i e_{i,1} \otimes a_i^{1/2}q_\alpha$, that we see in $M_n(\cM)$. Observe that $v_\alpha^* v_\alpha = \sum_i e_{1,1}\otimes q_\alpha a_i q_\alpha = e_{1,1} \otimes q_\alpha$, so $v_\alpha$ is a partial isometry. It is well-known that the projections $e_{1,1} \otimes 1$ and $1_n \otimes 1$ are equivalent in $M_n(\cM)$. Indeed, it follows from a repeated use of \cite[Proposition V.1.36]{Takesaki1} that there are isometries $u_i\in \cM$ (that is $u_i^* u_i=1$) such that $1 = \sum_{i=1}^n u_i u_i^*$. Then $u:= \sum_i e_{1,i} \otimes u_i \in \cM$ realizes the equivalence between $u^* u = 1_n \otimes 1$ and $u u^*= e_{1,1} \otimes 1$. So by the properties of $q_\alpha$ given in Lemma~\ref{lem:typeIII}, we have
\[ e_{1,1} \otimes 1 - v_\alpha^* v_\alpha =e_{1,1} \otimes (1-q_\alpha)\sim e_{1,1} \otimes 1 \sim 1\textrm{ in }M_n(\cM).\]
In particular, we have
\[ 1_n \otimes 1 - v_\alpha v_\alpha^* \prec e_{1,1} \otimes 1 - v_\alpha^* v_\alpha,\]
and there is $w_\alpha \in M_n(\cM)$ such that $w_\alpha w_\alpha^* = 1_n \otimes 1 - v_\alpha v_\alpha^*$ and $w_\alpha^* w_\alpha \leq e_{1,1} \otimes 1 - v_\alpha^* v_\alpha$. Letting $u_\alpha = v_\alpha + w_\alpha$, we therefore have
\[ u_\alpha^* u_\alpha \leq e_{1,1}\otimes 1,\ \  u_\alpha u_\alpha^* = 1_n \otimes 1.\]
We can therefore define $p_{i,\alpha}\in \cM$ by $u_\alpha^* (e_{i,i}\otimes 1) u_\alpha=e_{1,1}\otimes p_{i,\alpha}$. The fact that $u_\alpha u_\alpha^*=1$ implies that $p_{i,\alpha}$ are pairwise orthogonal projections, but a priori we only have $\sum_i p_{i,\alpha} \leq 1$. However, the sum is close to $1$ as
\begin{equation}\label{eq:Pi_almost_sum_to_1} e_{1,1}\otimes (\sum_i p_{i,\alpha}) = u_\alpha^* u_\alpha \geq v_\alpha^* v_\alpha = e_{1,1} \otimes q_\alpha.\end{equation}
Moreover, by the definition of $p_{i,\alpha}$, we have
\begin{align*} e_{1,1} \otimes q_\alpha p_{i,\alpha} q_\alpha &= v_\alpha^* v_\alpha u_\alpha^* (e_{i,i} \otimes 1) u_\alpha v_\alpha^* v_\alpha\\
  &= v_\alpha^* (e_{i,i} \otimes 1) v_\alpha\\
  &= e_{1,1} \otimes q_\alpha a_i q_\alpha.\end{align*}
That is,
\begin{equation}\label{eq:qAq=qPq} q_\alpha p_{i,\alpha} q_\alpha = q_\alpha a_i q_\alpha.\end{equation}

As in the finite case, let us denote by $\|\cdot\|_\varphi$ the norm on $\cM^n$ given by \[ \| (b_i)_i \|_\varphi^2 = \sum_{i=1}^n \varphi(b_i^* b_i).\]
Remembering \eqref{eq:qAq=qPq}, we can decompose $a_i - p_{i,\alpha}  = a_i - q_\alpha a_i q_\alpha -(1-q_\alpha) p_{i,\alpha}q_\alpha - p_{i,\alpha}(1 - q_\alpha)$ and obtain by the triangle inequality
\[ \| (a_i - p_{i,\alpha})_i \|_\varphi \leq \|(a_i - q_\alpha a_i q_\alpha)_i\|_\varphi+\|( (1-q_\alpha) p_{i,\alpha}q_\alpha )_i\|_\varphi + \| ( p_{i,\alpha}(1 - q_\alpha) )_i\|_\varphi.\]
It follows from Lemma~\ref{lem:semifinite-to-finite} that the first term goes to $0$ as $\alpha \to \infty$. The last term is straightforward to bound:
\[ \| ( p_{i,\alpha}(1 - q_\alpha) )_i\|_\varphi^2 = \varphi( (1-q_\alpha) (\sum_i p_{i,\alpha}) (1-q_\alpha))\leq \varphi(1-q_\alpha) \to 0.\]
The middle term is bounded as follows
\begin{align*}
\| ((1-q_\alpha) p_{i,\alpha}q_\alpha)_i\|_\varphi^2 &= \sum_i \varphi(q_\alpha p_{i,\alpha}(1-q_\alpha)  p_{i,\alpha} q_\alpha)\\
& = \sum_i \varphi(q_\alpha a_i q_\alpha -(q_\alpha a_i q_\alpha)^2 ),
\end{align*}
which goes to $\sum_i \varphi(a_i - a_i^2)<\varepsilon$. All in all, this implies that 
\[ \limsup_\alpha \| (a_i - p_{i,\alpha})_i \|_\varphi <\sqrt{\varepsilon}.\]
We are not completely done yet, as $p_{i,\alpha}$ do not sum to $1$. But almost. Indeed, by \eqref{eq:Pi_almost_sum_to_1}, we have $\sum_i p_{i,\alpha}$ converges to $1$ and in particular 
\[ \lim_\alpha \varphi( \sum_i p_{i,\alpha}) = 1.\]
This implies that if we replace $p_{1,\alpha}$ by $p_{1,\alpha}+(1-\sum_i p_{i,\alpha})$, we obtain a PVM in $\cM$ which still satisfies
\[ \limsup_\alpha \| (a_i - p_{i,\alpha})_i \|_\varphi <\sqrt{\varepsilon}.\]
This concludes the proof of the Theorem in the type III case.

\section{Almost commuting PVMs are close to commuting PVMS}\label{sec:AC}

This short section is devoted to the proof of Theorem~\ref{thm:pvm_almost_commute}.

Denote by $(a_i)$ the POVM $a_i = \sum_j q_j p_i q_j$. We can compute 
\begin{align*}
\varepsilon & > \sum_{i,j} \|p_i q_j - q_j p_i\|_\varphi^2\\
& = \varphi(\sum_{i,j} p_i q_j p_i + q_j p_i q_j - (p_i q_j)^2 - (q_j p_i)^2)\\
& = 2 - 2\Re(\sum_i \varphi(a_i p_i))\\
& = \sum_i \|p_i - a_i\|_\varphi^2 + (1 - \varphi(\sum_i a_i^2)).
\end{align*}
We can apply Theorem~\ref{thm:orthonormalization} to the $a_i$ in the von Neumann algebra $\cN$ generated by the $a_i$'s, and obtain a PVM $p'_i$ belonging to $\cN$ such that
\[ \sum_i \|a_i - p'_i\|_\varphi^2 < 9 (\varepsilon - \sum_i \|p_i - a_i\|_\varphi^2).\]
But since $a_i$ belongs to the commutant of $\{q_j\}$, the same is true for $\cN$, so $[p'_i,q_j]=0$ for all $i,j$. Using the triangle inequality and the Cauchy-Schwarz inequality, we conclude as follows
\begin{align*}
\left(\sum_i \|p_i - p'_i\|_\varphi^2\right)^{\frac 1 2} & \leq \left(\sum_i \|p_i - a_i\|_\varphi^2\right)^{\frac 1 2} + \left(\sum_i \|a_i - p'_i\|_\varphi^2\right)^{\frac 1 2} \\
& \leq \sqrt{10} \left(\sum_i \|p_i - a_i\|_\varphi^2+ \frac 1 9 \|a_i - p'_i\|_\varphi^2\right)^{\frac 1 2}\\
& < \sqrt{10 \varepsilon}.
\end{align*}

\section{Hahn-Banach}\label{sec:HB}

We conclude this note with a quite unrelated subject, except that it is also an infinite dimensional generalization of a key result in \cite{quantumsoundness}, and that is is also used in \cite{quantumsoundness2}.

Lemma 9.2 in \cite{quantumsoundness} states that for any finite collection $a_1,\dots,a_n$ of positive matrices, 
\begin{equation}\label{eq:minmaxfd} \min\{ \tr(z) \mid z \geq a_i \forall i\} = \max\{ \sum_i \tr(a_i t_i) \mid 0 \leq t_i \leq 1, \sum_i t_i=1\},\end{equation}
and that moreover any pair of minimizer $z$ and maximizer $(t_1,\dots,t_n)$ satisfies 
\begin{equation}\label{eq:formula}z = \sum_i t_i a_i.\end{equation}

The equality in \eqref{eq:minmaxfd} is known to be true more generally in any semifinite von Neumann algebra, as a particular case of a duality for Pisier's operator-space valued non-commutative $L_p$ spaces $L_p(\cM;\ell_\infty)$ \cite{MR1648908}. Formula \eqref{eq:minmaxfd} corresponds to the case $p=1$ and $\cM=M_n(\C)$ in \cite[Proposition 2.1 (iii)]{MR2276775}. To the author's knowledge, \eqref{eq:formula} has not been observed or used earlier in operator space theory.

It turns out that the preceding is true more generally in arbitrary von Neumann algebras, as follows. For $t \in \cM$ and $\varphi \in \cM_*$, we use the standard notation $t\varphi \in \cM_*$ and $\varphi t \in \cM_*$ to denote the linear forms $x \mapsto \varphi(xt)$ and $x \mapsto \varphi(tx)$ respectively.
\begin{proposition}\label{prop:minmax} Let $\cM$ be a von Neumann algebra, and $\varphi_1, \dots , \varphi_n \in (\cM_*)_+$ be normal positive linear forms. Then 
\begin{equation}\label{eq:minmax}
\inf\{ \psi(1) \mid \psi \in \cM_*, \psi \geq \varphi_i \forall i\} = \sup\{ \sum_i \varphi_i(t_i) \mid (t_i) \textrm{POVM in }\cM\}.
\end{equation} 
Moreover, the infimum and the supremum are both attained, and any pair of a minimizer $\psi$ and a maximizer $(t_1,\dots,t_n)$ satisfies $t_i \psi = t_i \varphi_i$, $\psi t_i  = \varphi_it_i $and $\psi = \sum_i t_i \varphi_i= \sum_i \varphi_i t_i$.
\end{proposition}
\begin{corollary}
For every $\varphi_1,\dots,\varphi_n \in (\cM_*)_+$, there is a unique element of $\cM_*$ of minimal norm such that $\psi \geq \varphi_i$ for all $i$.

Moreover, there is a POVM $t_1,\dots,t_n \in \cM$ such that $\psi= \sum_i t_i \varphi_i$.
\end{corollary}
\begin{proof}[Proof of Proposition~\ref{prop:minmax}]
The inequality $\geq$ in \eqref{eq:minmax} is clear: if $t_1,\dots,t_n \in \cM$ is any POVM and $\psi \geq \varphi_i$ for all $i$, then
\[ \sum_i \varphi_i(t_i) \leq \sum_i \psi(t_i) = \psi(1).\]
The converse relies on Hahn-Banach. We rather use the variant given in \cite[Lemma A.16]{pisierTP}. Define 
\[ m := \inf\{ \psi(1) \mid \psi \in \cM_*, \psi \geq \varphi_i \forall i\}.\]
Consider the weak-* closed convex subset of $\cM^{n+1}$
\[ S= \{(t_0,\dots,t_n) \in \cM^{n+1} \mid 0 \leq t_i \leq 1 \forall i\},\]
and for every self-adjoint $\psi \in \cM_*$, define $f_\psi \in \ell_\infty(S)$ by
\[ f_\psi(t_0,\dots,t_n) = \psi(t_0) - m + \sum_{i=1}^n (\varphi_i-\psi)(t_i).\]
We claim that $\sup_S f_\psi \geq 0$. Indeed, it is a general fact that if a self-adjoint element $\rho \in \cM_*$ has Jordan decomposition $\rho = \rho_+-\rho_-$ (see \cite[Theorem III.4.2]{Takesaki1}), then  $\sup_{0\leq t \leq 1} \rho(t) = \|\rho_+\|=\rho_+(1)$. In our situation, we obtain
\[ \sup_{S} f_\psi = \left(\psi_+ + \sum_{i=1}^n  (\varphi_i-\psi)_+\right)(1) - m.\]
For every $1 \leq j \leq n$, using that $\psi_+ \geq \psi$, $(\varphi_j - \psi)_+ \geq \varphi_j - \psi$ and $(\varphi_i - \psi)_+\geq 0$ for $i \neq j$, we see that
\[ \psi_+ + \sum_{i=1}^n  (\varphi_i-\psi)_+ \geq \varphi_j.\]
By the definition of $m$, this implies $\sup_{S} f_\psi \geq 0$ as claimed.

Consider now the convex cone $\mathcal F\subset \ell_\infty(S)$ generated by the \emph{convex} set $\{f_\psi \mid \psi = \psi^* \in \cM_*\}$~:
\[ \mathcal{F} = \{\lambda f_\psi \mid \psi = \psi^* \in \cM_*, \lambda \in (0,\infty)\}.\]
The elements of $\mathcal{F}$ are affine weak-* continuous maps on $S$, and we have just proved that $\forall f \in \mathcal{F}$, $\sup_{S} f \geq 0$. We can therefore apply \cite[Lemma A.16]{pisierTP} and obtain  $(t_0,\dots,t_n) \in S$ such that $f_\psi(t_0,\dots,t_n) \geq 0$ for every self-adjoint $\psi \in \cM_*$. Equivalently,
\[ \psi(t_0 - \sum_1^n t_i) + \sum_1^n \varphi_i(t_i) \geq m.\]
This implies that $t_0=\sum_1^n t_i$, and that $\sum_1^n \varphi_i(t_i) \geq m$. In other words, we have obtained positive elements $(t_1,\dots,t_n)\in \cM$ such that $\sum_i t_i \leq 1$ and $\sum_i \varphi_i(t_i) \geq m$. A fortiori (say replacing $t_n$ by $t_n + (1-\sum_1^n t_i)$), there are positive $t_i$ with $\sum_i t_i = 1$ and $\sum_i \varphi_i(t_i) \geq m$. This proves at the same time the inequality $\leq$ in \eqref{eq:minmax} and that the supremum in \eqref{eq:minmax} is attained. 

Let us justify that the infimum is also attained. By the weak-* compactness of the unit ball of $\cM^*$, we have that the infimum of $\psi(1)$ over all $\psi \in \cM^*$ such that $\psi \geq \varphi_i$ for all $i$ is attained at some $\psi \in \cM^*$. But using that $\cM_*$ is $L$-embedded in $\cM^*$ \cite[Theorem III.2.14]{Takesaki1}, we obtain that $\psi$ necessarily belongs to $\cM_*$.

Consider now $\psi$ attaining the infimum, and $(t_1,\dots,t_n)$ attaining the supremum in \eqref{eq:minmax}. We then have
\[ \sum_i (\psi-\varphi_i)(t_i)=0.\]
This implies (since $(\psi-\varphi_i)(t_i) \geq 0$ is clear) that $(\psi-\varphi_i)(t_i) = 0$ for all $i$. By the Cauchy-Schwarz inequality, we obtain that for every $x \in \cM$,
\[ |(\psi - \varphi_i)(xt_i)|^2 = |(\psi - \varphi_i)(xt_i^{1/2} t_i^{1/2})|^2 \leq (\psi - \varphi_i)(t_i) (\psi - \varphi_i)(x t_i x^*) = 0.\]
Hence $t_i (\psi-\varphi_i)=0$. Summing over $i$ we also obtain $\sum_i t_i \varphi_i = \sum_i t_i \psi = \psi$. Taking the adjoints we deduce $\psi t_i = \varphi_i t_i$.
\end{proof}

\end{document}